\documentclass[a4paper]{amsart}
\usepackage{graphicx}
\usepackage{amssymb}
\usepackage{amsmath}
\usepackage{amsthm}
\usepackage{amscd}
\usepackage[all,2cell]{xy}

\usepackage[pagebackref,colorlinks]{hyperref}

\usepackage{multicol}

\usepackage{pifont}
\usepackage{float}
\usepackage{tikz-cd}
\allowdisplaybreaks[4]

\usepackage{array, tabularx}

\UseAllTwocells \SilentMatrices
\newtheorem{thm}{Theorem}[section]

\newtheorem{cor}[thm]{Corollary}

\newtheorem{lem}[thm]{Lemma}
\newtheorem{exm}[thm]{Example}

\theoremstyle{definition}
\newtheorem{defn}[thm]{Definition}
\theoremstyle{remark}
\newtheorem{rem}[thm]{\bf Remark}
\numberwithin{equation}{section}

\newcommand{\xrto}{\xrightarrow}

\begin{document}
\title[Support $\tau$-tilting modules over trivial extensions]{Support $\tau$-tilting modules over trivial extensions of hereditary algebras}
\author[Rong,Li] {Rong Rong, Zhi-Wei Li$^*$}

\makeatletter
\@namedef{subjclassname@2020}{\textup{2020} Mathematics Subject Classification}
\makeatother

\thanks{$^*$ The corresponding author}
\date{\today}

\thanks{rongty$\symbol{64}$xzit.edu.cn, zhiweili@jsnu.edu.cn}

\maketitle

\dedicatory{}%
\commby{}%

\begin{abstract}
Let $A$ be a finite-dimensional basic hereditary algebra and let
$T(A)=A\ltimes D(A)$ be its trivial extension. Building on the
classification of indecomposable $\tau$-rigid $T(A)$-modules, we give
explicit Hom-vanishing conditions characterizing arbitrary basic
$\tau$-rigid $T(A)$-modules. For such a module $M$, we also determine
its maximal projective complement in terms of the support of the
underlying $A$-module $U(M)$, and hence obtain an explicit criterion
for $M$ to be support $\tau$-tilting.

As an application, for the linearly oriented quiver of type $A_n$,
we classify all basic rank-two $\tau$-rigid modules over
$T(\Bbbk A_n)$ and prove that their number is
\[
\binom{n}{2}\binom{n+1}{2}.
\]
We also show that every basic $\tau$-tilting
$T(\Bbbk A_n)$-module contains an indecomposable projective direct
summand. Finally, for two orientations of a quiver of type $D_4$,
we determine the corresponding support $\tau$-tilting compatibility
graphs and their face distributions, from which we obtain and compare the associated F-triangles.
\end{abstract}

\section{Introduction}

Since its inception in the foundational work of Adachi-Iyama-Reiten [AIR], $\tau$-tilting theory has emerged as one of the most active and fruitful areas in the representation theory of finite-dimensional algebras. As a natural generalization of classical tilting theory, $\tau$-tilting theory provides a complete framework for mutation. The central objects of the theory are support $\tau$-tilting modules. An almost complete support $\tau$-tilting module always has exactly two complements, while an almost complete tilting module may have zero, one or two complements. 

Support $\tau$-tilting modules have been classified for several classes of algebras. Preprojective algebras were classified by Buan–Iyama–Reiten–Scott \cite{BIRS} for the non-Dynkin case and by Mizuno \cite{Mizuno} for the Dynkin case. Adachi \cite{A} obtained the classification for Nakayama algebras. Iyama and Zhang \cite{IZ} classified them over the Auslander algebra of $\Bbbk[x]/(x^n)$. Zito \cite{Zito} treated the case of cluster-tilted algebras. 

For lower triangular matrix rings, Gao and Huang \cite{GH} initiated the study, which was later extended by Peng, Ma and Huang \cite{PMH} and further generalized by Zhang \cite{Zhang} via recollements of abelian categories. Using adjoint functor techniques, Li and Zhang constructed two classes of support 
$\tau$-tilting modules over trivial extensions by bimodules, generalizing and extending previous results on support $\tau$-tilting modules over triangular matrix rings \cite{LiZhang}. More recently, Wang and Zhang \cite{WZ} studied support $\tau$-tilting modules over trivial extensions of gentle tree algebras, and proved that the number of such modules depends only on the number of simple modules of the underlying gentle tree algebra.

 Recall that, for a finite-dimensional algebra $A$ over a field $\Bbbk$, the {\itshape trivial extension} $T(A)=A\ltimes D(A)$ of $A$ is the $\Bbbk$-algebra with multiplication given by 
$$(a,f)(b,g)=(ab,ag+fb).$$
Here, $D(A)=\mathrm{Hom}_\Bbbk(A,\Bbbk)$ is the dual $A$-bimodule and the bimodule structure is given such that for each $a\in A$ and $f\in D(A)$, $(af)(x)=f(xa)$ and $(fa)(x)=f(ax)$ for all $x\in A$. The trivial extension is an important construction in homological algebra and representation theory. It not only provides a universal passage from a finite-dimensional algebra to a class of symmetric algebras \cite{Tachikawa80,FGR}, but also serves as a bridge connecting tilting theory, cluster theory, and derived categories in modern representation theory \cite{HW, Happel,AHR, ABS}. In \cite{Tachikawa80}, Tachikawa classified the indecomposable representations of the trivial extension over hereditary algebras of finite representation type. Motivated by Tachikawa's work, Li, Xu and Zhao gave a complete determination of $\tau$-rigid objects of a trivial extension of a hereditary abelian category $\mathcal{H}$ by some right exact functor $F$ in terms of data in $\mathcal{H}$ \cite{LXZ}. 

 The aim of this paper is to give a systematic translation of those conditions of \cite[Theorem 1.1]{LXZ} into the four-family decomposition; the compatibility formulation for arbitrary basic sums of the four families; the explicit maximal projective complement $W_M=\bigoplus_{j\in J(M)}P_j$, where $J(M)=\{j\in \{1,2,\cdots,|A|\}\ | \ j\notin \operatorname{supp}U(M)\}\}$. 
 
 \vskip 5pt

\noindent{\bf Main Theorem.} (Theorem \ref{thm:main result})  {Let $M$ be a basic $T(A)$-module with a decomposition
$$M=Z(X_0)\oplus Z(R) \oplus L(S)\oplus \bigoplus_{i=1}^t\Omega Z(Y_i),$$
where $X_0$ is an $A$-module without projective direct summands, $R,S\in A\mbox{-}\mathrm{proj}$ and each $\Omega Z(Y_i)=(P_i\oplus \nu(Q_i), \left(\begin{smallmatrix}
	0&0\\
	\nu(\partial_i) & 0
\end{smallmatrix}\right))$ is a minimal syzygy $T(A)$-module associated with a minimal projective resolution $0\to P_i\xrto{\partial_i}Q_i\to Y_i\to 0$. Let $0\to R_1\to R_0\to \nu(R)\to 0$ and $0\to P_{1i}\to P_{0i}\to \tau(Y_i)\to 0$ be minimal projective resolutions. Then 

$\textup{(a)}$ $M$ is a basic $\tau$-rigid $T(A)$-module if and only if the following conditions hold.
\begin{enumerate}

\item $\mathrm{Hom}_A(X_0\oplus R\oplus S\oplus \bigoplus_{i=1}^tP_i, \tau(X_0))=0$.

\item $\mathrm{Hom}_A(\bigoplus_{i=1}^tY_i,\bigoplus_{i=1}^t\tau(Y_i)\oplus \nu(R))=0$.

\item $$\mathrm{Hom}_A(R_0\oplus \bigoplus_{i=1}^tP_{0i}, X_0\oplus R\oplus S\oplus \bigoplus_{i=1}^tP_i)=0$$ and $$\mathrm{Hom}_A( R_1\oplus \bigoplus_{i=1}^tP_{1i}, \nu(S)\oplus \bigoplus_{i=1}^t\nu(Q_i))=0.$$
\end{enumerate}

$\textup{(b)}$ $M$ is support $\tau$-tilting if and only if $M$ is $\tau$-rigid and $|M|+|J(M)|=|A|$. In that case, $(M,L(W_M))$ is the associated support $\tau$-tilting pair.}
\vskip 5pt

As applications, we give the uniform rank-two $\tau$-rigid modules classification in type $A_n$ and a formula for the number of basic rank-two $\tau$-rigid $T(\Bbbk A_n)$-modules $\binom{n}{2}\binom{n+1}{2}$;  the refined $D_4$ compatibility $F$-triangle calculations.

\section{Preliminaries}

We begin by reviewing the relevant definitions and properties of $\tau$-tilting theory and of trivial extensions of algebras that will be used throughout the paper. 
We shall let $\Bbbk$ be a field, and $A$ a finite-dimensional basic $\Bbbk$-algebra. We denote by $A\mbox{-}\mathrm{mod}$ the category of finite-dimensional left $A$-modules. 

\subsection{The module category of the trivial extension of $A$} Let $$\nu:=DA\otimes_A-\colon A\mbox{-}\mathrm{mod}\to A\mbox{-}\mathrm{mod}$$ be the Nakayama functor. Denote by $A\mbox{-}\mathrm{proj}$ and $A\mbox{-}\mathrm{inj}$ the subcategories of $A\mbox{-}\mathrm{mod}$ consisting of the finite-dimensional projective and injective left $A$-modules respectively. Then $\nu, \nu^{-}:=\mathrm{Hom}_A(DA,-)$ are mutually inverse equivalences between $A\mbox{-}\mathrm{proj}$ and $A\mbox{-}\mathrm{inj}$. 
If $A$ is a finite-dimensional hereditary algebra, then $\nu(X)=0$ for every non-projective module $X$; see \cite[Theorem VII.1.4(c)]{ASS}.

By \cite[Section 1]{FGR},  a left $T(A)$-module can be viewed as a pair $(X,\alpha)$, where $X$ is a left $A$-module and $\alpha\colon \nu(X)\to X$ is a left $A$-module homomorphism such that the composition $\nu^2(X)\xrto{\nu(\alpha)}\nu(X)\xrto{\alpha}X$ is zero. A morphism $f\colon (X,\alpha)\to (Y,\beta)$ is a morphism $f\colon X\to Y$ in $A\mbox{-}\mathrm{mod}$ such that the diagram 
\[
\xymatrix{\nu(X)\ar[r]^\alpha\ar[d]_{\nu(f)} &X\ar[d]^f\\
\nu(Y) \ar[r]^\beta& Y}
\]
is commutative. Composition is just composition in $A\mbox{-}\mathrm{mod}$.

By \cite[Proposition 1.3]{FGR}, there exist adjoint pairs $(L,U)$ and $(C,Z)$:
\[
\xymatrix{A\mbox{-}\mathrm{mod} \ar@<.5ex>[r]^-{L}&T(A)\mbox{-}\mathrm{mod} \ar@<.5ex>[l]^-{U}\ar@<.5ex>[r]^-{C}& A\mbox{-}\mathrm{mod}.\ar@<.5ex>[l]^-{Z}}
\]
Here, the functor $L$ is defined by $L(X)=(X\oplus \nu(X),\left(\begin{smallmatrix}
	0&0\\
	1&0
\end{smallmatrix}\right))$,
$L(f)=\left(\begin{smallmatrix}
	f&0\\
	0&\nu(f)
\end{smallmatrix}\right)$, for $f\in \mathrm{Hom}_A(X,Y).$
The functor $U$ is defined by $U(X,\alpha)=X$ and $U(f)=f$. The functor $Z$ is given by $Z(X)=(X,0)$ and $Z(f)=f$. The functor $C$ is defined by the cokernel
$C(X,\alpha)=\mathrm{cok}\alpha$ and $C(f)$ is the induced morphism. According to \cite[Corollary 1.6(a)]{FGR}, the functors $L,C$ are right exact, and $U,Z$ are exact. Moreover, by \cite[Corollary 1.6(c)]{FGR}, the projective objects of $T(A)\mbox{-}\mathrm{mod}$ are precisely the objects $L(P)$ with $P\in A\mbox{-}\mathrm{proj}$. 

 Let $(X,\alpha)\in T(A)\mbox{-}\mathrm{mod}$, and let $p\colon P\to \mathrm{cok}\alpha$ be a projective cover in $A\mbox{-}\mathrm{mod}$. From the exact sequence 
$$\nu(X)\xrto{\alpha}X\xrto{\pi}\mathrm{cok}\alpha\to 0,$$ the projectivity of $P$ yields a morphism $q\colon P\to X$ such that $\pi\circ q=p$. By \cite[Corollary 1.7 (a)]{FGR}, the morphism $(p,\alpha \circ \nu(q))\colon L(P)\to (X,\alpha)$ is a projective cover in $T(A)\mbox{-}\mathrm{mod}$, as illustrated in the following commutative diagram:
\[
\xymatrix{\nu(P)\oplus \nu^2(P)\ar[r]^-{\left(\begin{smallmatrix}0&0\\
1&0\end{smallmatrix}\right)}\ar[d]_-{(\nu(q),\nu(\alpha)\circ \nu^2(q))}&P\oplus \nu(P)\ar[d]_-{(q,\alpha \circ\nu(q))}\ar[r]^-{(1,0)} & P \ar[d]^-{p}\ar@{.>}[ld]_-{q}\ar[r]&0\\
\nu(X)\ar[r]^-\alpha & X \ar[r]^-{\pi_X} & \mathrm{cok}\alpha \ar[r]& 0.}
\]
Furthermore, for any object $X$ in $A\mbox{-}\mathrm{mod}$ with minimal projective resolution $0\to P_1\xrto{\partial} P_0\xrto{p} X\to 0$, the morphism $L(P_0)\xrto{(p,0)}Z(X)$ is a projective cover of $Z(X)$ in $T(A)\mbox{-}\mathrm{mod}$. In particular, the syzygy $\Omega Z(X)$ of $Z(X)$ has the explicit form: $(P_1\oplus \nu(P_0),\left(\begin{smallmatrix}
	0&0\\
	\nu(\partial) & 0
\end{smallmatrix}\right))$.

\begin{lem} \label{lem:ind of TA} \textup{\cite[Corollary 2.4]{LXZ}} Let $A$ be a finite-dimensional hereditary algebra. A left $T(A)$-module is indecomposable if and only if it is isomorphic to one of:
	\begin{enumerate}
	\item  $Z(X)$,  where $X$ is an indecomposable $A$-module;
	
	\item $L(P)$, where $P$ is an indecomposable projective $A$-module;
	
	\item $\Omega Z(X)=(P\oplus \nu(Q),	\left(\begin{smallmatrix}
	0&0\\
	\nu(\partial) & 0
\end{smallmatrix}\right))$, where $X$ is a non-projective indecomposable left $A$-module and $0\to P\xrto{\partial}Q\to X\to 0$ is a minimal projective resolution.
\end{enumerate}
\end{lem}
 For simplicity, we call the indecomposable $T(A)$-module $\Omega Z(X)$ above a {\it minimal syzygy module}.

\subsection{The support $\tau$-tilting modules} We recall the definition of support $\tau$-tilting modules in \cite{AIR}. Let $A$ be a finite-dimensional $\Bbbk$-algebra and $\tau$ the Auslander-Reiten translation in $A\mbox{-}\mathrm{mod}$ \cite{ASS}. For a left $A$-module $M$, we use $|M|$ to denote the number of isomorphism classes of indecomposable direct summands of $M$.

\begin{defn} \textup{\cite[Definition 0.1]{AIR}} Let $M$ be a left $A$-module in $A\mbox{-}\mathrm{mod}$.
\begin{enumerate}
\item We call $M$ {\itshape $\tau$-rigid} if $\mathrm{Hom}_A(M,\tau(M))=0$.

\item We call a pair $(M,P)$ a {\itshape $\tau$-rigid} pair if $M$ is $\tau$-rigid and $P\in A\mbox{-}\mathrm{proj}$ such that $\mathrm{Hom}_A(P,M)=0$.

\item We call a pair $(M,P)$ a {\itshape support $\tau$-tilting} pair if $(M,P)$ is a $\tau$-rigid pair such that $|M|+|P|=|A|$.

\item We call $M$ {\itshape support $\tau$-tilting} if there is $P\in A\mbox{-}\mathrm{proj}$ such that $(M,P)$ is a support $\tau$-tilting pair.

\item We call $M$ {\itshape $\tau$-tilting} if $(M,0)$ is a support $\tau$-tilting pair.
\end{enumerate}
\end{defn}
We recall that a module is said to be {\itshape basic} provided it is a direct sum of pairwise non-isomorphic indecomposable modules. By \cite[Proposition 2.3]{AIR}, we know that $M$ determines $P$ uniquely in a basic support $\tau$-tilting pair $(M,P)$. Here, we call $(M,P)$ {\itshape basic} if $M$ and $P$ are basic. Thus we can identify basic support $\tau$-tilting modules with basic support $\tau$-tilting pairs. We use $\tau\mbox{-}\mathsf{tilt}~ A$ (respectively $s\tau\mbox{-}\mathsf{tilt}~A$) to denote the set of isomorphism classes of basic $\tau$-tilting $A$-modules (respectively support $\tau$-tilting $A$-modules).

\section{The main result}
In this section, we give a proof of our main result.

The following result classifies the $\tau$-rigid modules of the trivial extension $T(A)=A\ltimes D(A)$ of a finite-dimensional basic hereditary $\Bbbk$-algebra $A$.

\begin{lem}\label{lem:general tau-rigid} \textup{\cite[Theorem 1.1]{LXZ}}
A left $T(A)$-module $M$ is $\tau$-rigid if and only if it is isomorphic to $$Z(X)\oplus  (P\oplus I, \left(\begin{smallmatrix}
	0&0\\
	\alpha&0
\end{smallmatrix}\right)),$$
where $P$ is a projective $A$-module and $I$ is an injective $A$-module such that $\alpha$ is an epimorphism, and the following conditions hold:
	\begin{enumerate}

	\item  $\mathrm{Hom}_A(X\oplus P,\tau(X))=0$;

\item $\mathrm{Hom}_A(\mathrm{ker} \alpha, \tau(\mathrm{ker} \alpha)\oplus \tau(\nu(X))))=0$;

\item $\mathrm{Hom}_A(R_0\oplus P_0,X\oplus P)=0=\mathrm{Hom}_A(R_1\oplus P_1, I)$, where $0\to R_1\to R_0\to \nu(X)\to 0$ and $0\to P_1\to P_0\to \mathrm{ker}\alpha\to 0$ are projective resolutions of $\nu(X)$ and $\mathrm{ker}\alpha$ respectively.
\end{enumerate}
\end{lem}

\begin{lem} \label{lem:ind of tau-rigid}\textup{\cite[Corollary 3.3]{LXZ}} Every indecomposable $\tau$-rigid $T(A)$-module belongs to exactly one of the following four mutually disjoint families.
\begin{enumerate}
\item Modules of type $Z(X)$, where $X$ is a non-projective indecomposable $\tau$-rigid $A$-module.

\item Modules of type $Z(R)$, where $R$ is an indecomposable projective $A$-module such that 
$\mathrm{Hom}_A(R_0,R)=0$ for a projective cover $R_0\to \nu(R)$.

\item Modules of type $L(S)$, where $S$ is an indecomposable projective $A$-module.

\item Minimal syzygy modules $\Omega Z(Y)=(P\oplus \nu(Q), \left(\begin{smallmatrix}
	0&0\\
	\nu(\partial) & 0
\end{smallmatrix}\right))$, where $P,Q$ are projective $A$-modules such that 
\begin{itemize}
\item[\ding{172}] $Y$ is a non-projective indecomposable $\tau$-rigid $A$-module,

\item[\ding{173}] $0\to P\xrto{\partial} Q\to Y\to 0$ is a minimal projective resolution,

\item[\ding{174}] $\mathrm{Hom}_A(P_0,P)=0=\mathrm{Hom}_A(P_1,\nu(Q))$, where $0\to P_1\to P_0\to \tau(Y)\to 0$ is a minimal projective resolution.
\end{itemize}
\end{enumerate}
\end{lem}




For a basic $T(A)$-module $M$, recall that there is an adjoint pair $(L,U)$, where $U\colon T(A)\mbox{-}\mathrm{mod}\to A\mbox{-}\mathrm{mod}$ is the forgetful functor. Let $P_1,\cdots, P_n$ be the representatives of the indecomposable projective $A$-modules, where $n=|A|$. Define 
$$\operatorname{supp}X:=\{j\ | \ \mathrm{Hom}_A(P_j,X)\neq 0\}.$$
Then define the set of admissible complement vertices by
$$J(M)=\{j\in \{1,2,\cdots,|A|\}\ | \ j\notin \operatorname{supp}U(M)\}\}$$
Then 
$$W_M=\bigoplus_{j\in J(M)}P_j$$
is the largest projective module such that $\mathrm{Hom}_{A}(W_M,U(M))=0$.

\begin{thm} \label{thm:main result} Let $M$ be a basic $T(A)$-module with a decomposition
$$M=Z(X_0)\oplus Z(R) \oplus L(S)\oplus \bigoplus_{i=1}^t\Omega Z(Y_i),$$
where $X_0$ is an $A$-module without projective direct summands, $R,S\in A\mbox{-}\mathrm{proj}$ and each $\Omega Z(Y_i)=(P_i\oplus \nu(Q_i), \left(\begin{smallmatrix}
	0&0\\
	\nu(\partial_i) & 0
\end{smallmatrix}\right))$ is a minimal syzygy $T(A)$-module associated with a minimal projective resolution $0\to P_i\xrto{\partial_i}Q_i\to Y_i\to 0$. Let $0\to R_1\to R_0\to \nu(R)\to 0$ and $0\to P_{1i}\to P_{0i}\to \tau(Y_i)\to 0$ be minimal projective resolutions. Then 

$\textup{(a)}$ $M$ is a basic $\tau$-rigid $T(A)$-module if and only if the following conditions hold.
\begin{enumerate}

\item $\mathrm{Hom}_A(X_0\oplus R\oplus S\oplus \bigoplus_{i=1}^tP_i, \tau(X_0))=0$.

\item $\mathrm{Hom}_A(\bigoplus_{i=1}^tY_i,\bigoplus_{i=1}^t\tau(Y_i)\oplus \nu(R))=0$.

\item $$\mathrm{Hom}_A(R_0\oplus \bigoplus_{i=1}^tP_{0i}, X_0\oplus R\oplus S\oplus \bigoplus_{i=1}^tP_i)=0$$ and $$\mathrm{Hom}_A( R_1\oplus \bigoplus_{i=1}^tP_{1i}, \nu(S)\oplus \bigoplus_{i=1}^t\nu(Q_i))=0.$$
\end{enumerate}

$\textup{(b)}$ $M$ is support $\tau$-tilting if and only if $M$ is $\tau$-rigid and $|M|+|J(M)|=|A|$. In that case, $(M,L(W_M))$ is the associated support $\tau$-tilting pair.
\end{thm}

\begin{proof}
$\textup{(a)}$
We apply Lemma~\ref{lem:general tau-rigid} to the decomposition of $M$. Set
$$
X=X_0\oplus R,\qquad
P=S\oplus\bigoplus_{i=1}^{t}P_i,
\qquad
I=\nu(S)\oplus\bigoplus_{i=1}^{t}\nu(Q_i),
$$
and let
$$
\alpha=\left(\begin{smallmatrix}1_{\nu(S)}&0\\
0&\bigoplus_{i=1}^t\nu(\partial_i)\end{smallmatrix}\right)
\colon
\nu(P)\longrightarrow I.
$$
With this notation, $M$ is of the form appearing in
Lemma~\ref{lem:general tau-rigid}.

Since $X_0$ has no projective direct summands and $A$ is
hereditary, \cite[Theorem VII.1.4(c)]{ASS} gives
$
\nu(X_0)=0.
$
Consequently,
$
\nu(X)=\nu(X_0\oplus R)\cong\nu(R).$
Moreover, since $R$ is projective,
$
\tau(X)=\tau(X_0\oplus R)\cong\tau(X_0).
$

We now translate the three conditions in
Lemma~\ref{lem:general tau-rigid}. For the first condition, we have $X\oplus P=X_0\oplus R\oplus S\oplus \bigoplus_{i=1}^tP_i$
and $\tau(X)=\tau(X_0)$. Therefore,
$\mathrm{Hom}_A(X\oplus P,\tau(X))=0$
is equivalent to 
$$\mathrm{Hom}_A(X_0\oplus R\oplus S\oplus \bigoplus_{i=1}^tP_i, \tau(X_0))=0$$
which is precisely condition \textup{(1)}.

Next, since $\mathrm{ker}1_{\nu(S)}=0$, we obtain
$\mathrm{ker}\alpha=\bigoplus_{i=1}^t\mathrm{ker}\nu(\partial_i).$
For every $i$, the minimal projective resolution
$0\to P_i\xrto{\partial_i}Q_i\to Y_i\to 0$
induces an isomorphism $\mathrm{ker}\nu(\partial_i)\cong \tau(Y_i)$
by \cite[Proposition IV.2.4(a)]{ASS}. Hence 
$$\mathrm{ker}\alpha\cong \bigoplus_{i=1}^t\tau(Y_i)=\tau(\bigoplus_{i=1}^tY_i).$$
Thus the second condition in Lemma~\ref{lem:general tau-rigid} becomes
$$\mathrm{Hom}_A(\bigoplus_{i=1}^t\tau(Y_i), \tau(\bigoplus_{j=1}^t\tau(Y_j)\oplus \nu(R)))=0.$$
Since $A$ is hereditary, every module has projective dimension at most one and injective dimension at most one. Moreover, each $Y_i$ is indecomposable and non-projective. Therefore, \cite[Corollary IV.2.15(b)]{ASS} yields 
$$\mathrm{Hom}_A(\tau(Y_i), \tau(\bigoplus_{j=1}^t\tau(Y_j)\oplus \nu(R)))\cong \mathrm{Hom}_A(Y_i, \bigoplus_{j=1}^t\tau(Y_j)\oplus \nu(R)).$$
Taking the direct sum over $i$, we obtain
$$\mathrm{Hom}_A(\bigoplus_{i=1}^t\tau(Y_i), \tau(\bigoplus_{j=1}^t\tau(Y_j)\oplus \nu(R)))=0\Longleftrightarrow \mathrm{Hom}_A(\bigoplus_{i=1}^tY_i, \bigoplus_{j=1}^t\tau(Y_j)\oplus \nu(R))=0.$$
This is condition \textup{(2)}. 

Finally, since $$\nu(X)=\nu(R)\quad \mbox{and}\quad \mathrm{ker}\alpha=\bigoplus_{i=1}^t\tau(Y_i),$$
 we may take $R_0, R_1$ from the minimal projective resolution of $\nu(R)$, and, by additivity, 
 $P_0=\bigoplus_{i=1}^tP_{0i}$ and $P_1=\bigoplus_{i=1}^tP_{1i}$
 for the projective resolution of $\mathrm{ker}\alpha$. Thus Lemma \ref{lem:general tau-rigid}(3) becomes
 $$\mathrm{Hom}_A(R_0\oplus \bigoplus_{i=1}^tP_{0i}, X_0\oplus R\oplus S\oplus \bigoplus_{i=1}^tP_i)=0$$
 and $$\mathrm{Hom}_A(R_1\oplus \bigoplus_{i=1}^tP_{1i}, \nu(S)\oplus \bigoplus_{i=1}^t\nu(Q_i))=0.$$
These are precisely the two equalities in condition \textup{(3)}. This proves part \textup{(a)}.

\medskip
\noindent $\textup{(b)}$ By Lemma~\ref{lem:ind of TA}, every basic projective $T(A)$-module is of the form $L(W)$, where $W$ is a basic projective $A$-module. Moreover, 
$|L(W)|=|W|.$
Thus $M$ is a basic support $\tau$-tilting $T(A)$-module if and only if there is some basic projective $A$-module $W$ such that $(M,L(W))$ is a basic support $\tau$-tilting pair.
Let $N=(E\oplus F, \left(\begin{smallmatrix} 0&0\\ \beta&0 \end{smallmatrix}\right))$ be a $T(A)$-module, where $\beta\colon \nu(E)\to F$. Then 
$$\mathrm{Hom}_{T(A)}(L(W),N)=\{\left(\begin{smallmatrix} u&0\\ v&\beta\circ \nu(u) \end{smallmatrix}\right) | u\in \mathrm{Hom}_A(W,E), v\in \mathrm{Hom}_A(W,F)\}.$$
Applying this description to the direct summands of $M$, we obtain
$$\mathrm{Hom}_{T(A)}(L(W),M)=0 \Longleftrightarrow \mathrm{Hom}_A(W,U(M))=0,$$
where $U(M)=X_0\oplus R\oplus S\oplus \nu(S)\oplus \bigoplus_{i=1}^t(P_i\oplus \nu(Q_i)).$
Write $W=\bigoplus_{j\in I}P_j$
for some $I\subseteq \{1,2,\cdots,|A|\}$. Then $$\mathrm{Hom}_A(W,U(M))=0\Longleftrightarrow I\subseteq J(M).$$
Equivalently, 
$$W\in \operatorname{add}W_M,\quad W_M=\bigoplus_{j\in J(M)}P_j.$$
In particular, $|W|\leq |W_M|=|J(M)|$
for every projective $W$ such that $$\mathrm{Hom}_{T(A)}(L(W),M)=0.$$

Suppose first that $M$ is a basic support $\tau$-tilting $T(A)$-module. Then there exists some basic projective $A$-module $W$ such that $(M,L(W))$ is a basic support $\tau$-tilting pair. Hence $M$ is $\tau$-rigid and 
$|M|+|W|=|A|.$
Since $W\in \operatorname{add}W_M$, we have 
$$|A|=|M|+|W|\leq |M|+|W_M|=|M|+|J(M)|.$$
On the other hand, $M$ is $\tau$-rigid and $\mathrm{Hom}_{T(A)}(L(W_M),M)=0$.
Therefore, $(M,L(W_M))$ is a basic $\tau$-rigid pair. The cardinality bound for basic $\tau$-rigid pairs gives 
$|M|+|J(M)|=|M|+|W_M|\leq |A|.$
Combining the two inequalities yields $|M|+|J(M)|=|A|.$

Conversely, suppose that $M$ is $\tau$-rigid and $|M|+|J(M)|=|A|$. By the definition of $J(M)$, 
$\mathrm{Hom}_A(W_M, U(M))=0,$
and hence $\mathrm{Hom}_{T(A)}(L(W_M),M)=0.$
Thus $(M,L(W_M))$ is a basic $\tau$-rigid pair. Moreover, 
$$|M|+|L(W_M)|=|M|+|W_M|=|M|+|J(M)|=|A|.$$
It follows that $(M,L(W_M))$ is a basic support $\tau$-tilting pair. Consequently, $M$ is a basic support $\tau$-tilting module and $(M,L(W_M))$ is the associated basic support $\tau$-tilting pair.
\end{proof}

\section{Applications to the linearly oriented quiver of type $A_n$}

In this section, we use Lemma \ref{lem:ind of tau-rigid} and Theorem \ref{thm:main result} to determine the basic $\tau$-rigid modules of the trivial extension of the linear quiver $A_n$.

\subsection{The compatibility graph of basic $\tau$-rigid modules}  Let $B$ be a finite-dimensional $\Bbbk$-algebra. Let $G_B$ be the graph whose vertices are the indecomposable $\tau$-rigid $B$-modules, with an edge
$$U\mathrel{\relbar} V$$
exactly when $U\oplus V$ is $\tau$-rigid.
 
 Because $$\mathrm{Hom}_{B}(\bigoplus_{i=}^rM_i, \tau(\bigoplus_{j=1}^sM_j))\cong \bigoplus_{i=1}^r\bigoplus_{j=1}^s\mathrm{Hom}_{B}(M_i,\tau(M_j)),$$
we have
$$\oplus_{i=1}^rM_i \in \tau\mbox{-}\mathsf{rigid}~B\Longleftrightarrow M_i\oplus M_j\in \tau\mbox{-}\mathsf{rigid}~B \ \mbox{for every} \ i\neq j.$$
Therefore
$$\mbox{basic rank-}r\ \tau \mbox{-rigid modules of}\ B\longleftrightarrow r\mbox{-cliques of}\ G_B.$$

 To avoid repetitions, we fix an ordering of the vertices and only add a vertex larger than all previously chosen vertices. After Theorem~\ref{thm:main result} gives the compatibility graph $G_B$, it is very convenient to determine the lower rank basic $\tau$-rigid $B$:
 \begin{enumerate}
   \item For each edge $\{U,V\}$, compute $N(U)\cap N(V),$
   where $N(U)$ is the set of larger compatible neighbors of $U$. Every $W\in N(U)\cap N(V)$ gives a triangle $\{U,V,W\}$.
   \item For each triangle ${U,V,W}$, compute
$N(U)\cap N(V)\cap N(W).$
Every vertex in this intersection gives a $4$-clique.
   \item Continue by intersecting the common neighborhoods of the vertices already chosen.
 \end{enumerate} 

The compatibility graph $G_B$ is precisely the
$1$-skeleton of the so-called \emph{$\tau$-rigid complex} of $B$ defined as follows \cite{DIJ}.
\begin{defn}
Let $B$ be a finite-dimensional basic $\Bbbk$-algebra. The
\emph{$\tau$-rigid complex} of $B$, denoted by
$\Delta_{\tau}(B)$, is the abstract simplicial complex defined as
follows. Its vertices are the isomorphism classes of indecomposable
$\tau$-rigid $B$-modules. A finite set
$$
{M_1,\ldots,M_r}
$$
of pairwise non-isomorphic indecomposable $\tau$-rigid
$B$-modules is a face of $\Delta_{\tau}(B)$ if and only if
$
M_1\oplus\cdots\oplus M_r
$
is a basic $\tau$-rigid $B$-module.
\end{defn}
Since 
$
M_1\oplus\cdots\oplus M_r
$
is $\tau$-rigid if and only if
$
M_i\oplus M_j
$
is $\tau$-rigid for every $i\neq j$. Hence the faces of
$\Delta_{\tau}(B)$ are exactly the cliques of $G_B$.
In particular, the $(r-1)$-dimensional faces of
$\Delta_{\tau}(B)$ are in bijection with the basic rank-$r$
$\tau$-rigid $B$-modules.

The complex $\Delta_{\tau}(B)$ is the module-vertex induced subcomplex of the \emph{support $\tau$-tilting complex} $\Delta_{\mathrm{s}\tau}(B)$ defined by Palu-Pilaud-Plamondon in \cite[Definition 1.8]{PPP}. Recall that $\Delta_{\mathrm{s}\tau}(B)$ is the abstract simplicial complex
whose vertices are of the following two types:
$$
(M,0),
\
M\in\operatorname{ind}\tau\text{-rigid}B, \qquad \text{and} \qquad
(0,P),
\
P\in\operatorname{ind}B\text{-proj}.
$$
Equivalently, one may denote the latter vertex by $P[1]$. A finite collection
$$
{M_1,\ldots,M_r,P_1[1],\ldots,P_s[1]}
$$
is a face of $\Delta_{\mathrm{s}\tau}(B)$ if and only if
$$
\left(
\bigoplus_{i=1}^{r}M_i,\bigoplus_{j=1}^{s}P_j
\right)
$$
is a basic $\tau$-rigid pair.

If $B$ has $n$ isomorphism classes of simple modules, then the
facets of $\Delta_{\mathrm{s}\tau}(B)$ are precisely the basic
support $\tau$-tilting pairs. Consequently, every facet contains $n$ vertices, and
$
\dim\Delta_{\mathrm{s}\tau}(B)=n-1.
$

The notation $L_i[1]$ distinguishes the projective-complement
vertex from the module vertex $L_i$, which itself is an
indecomposable $\tau$-rigid $B$-module.

Similar to the $\tau$-rigid compatibility graph, to avoid repetitions, we also fix an ordering of the vertices of $G_B^{\rm s\tau}$ and only add a vertex larger than all previously chosen vertices. The choice of an ordering on the vertices is not intrinsic to the compatibility graph or to the associated simplicial complex. We fix a total ordering only for computational purposes, so that each clique is constructed exactly once.

\subsection{The indecomposable $\tau$-rigid  $T(\Bbbk A_n)$-modules} We first give a general result for the trivial extension of the linear quiver $A_n$.
Let
$$
A_n:1\longrightarrow2\longrightarrow\cdots\longrightarrow n.
$$
For $1\leq a\leq b\leq n$, denote by
$
[a,b]
$
the indecomposable interval $\Bbbk A_n$-module supported on the vertices
$a,a+1,\ldots,b$. In particular,
$$S_a=[a,a],\qquad
P_a=[a,n]
\qquad\text{and}\qquad
I_b=[1,b].
$$

For interval modules, one has
\begin{equation}\label{eq:interval hom}
\mathrm{Hom}_{\Bbbk A_n}([a,b],[c,d])\neq0
\quad\Longleftrightarrow\quad
c\leq a\leq d\leq b.
\end{equation}

Moreover, if $b<n$, then
$
\tau[a,b]=[a+1,b+1],
$
and the minimal projective resolution of $[a,b]$ is
\begin{equation*}\label{eq:minimal proj resolu}
0\longrightarrow P_{b+1}
\longrightarrow P_a
\longrightarrow[a,b]\longrightarrow0.
\end{equation*}
In particular, every non-projective indecomposable $\Bbbk A_n$-module
$[a,b]$, with $b<n$, is $\tau$-rigid.

We determine the indecomposable $\tau$-rigid
$T(\Bbbk A_n)$-modules occurring in Lemma \ref{lem:ind of tau-rigid}.

First, every non-projective indecomposable interval module gives an
indecomposable $\tau$-rigid module
$
Z([a,b]),
\qquad
1\leq a\leq b<n.$

Next, for $R=P_r$, we have
$
\nu(P_r)=I_r=[1,r].
$
The projective cover of $I_r$ is $P_1\to I_r$. Hence the
condition for $Z(P_r)$ to be $\tau$-rigid is
$
\mathrm{Hom}_{\Bbbk A_n}(P_1,P_r)=0.
$
By \eqref{eq:interval hom}, this holds precisely for $2\leq r\leq n$.
Thus the modules of this type are
$
Z(P_r),
\qquad
2\leq r\leq n.
$

Every module
$
L(P_r),
\qquad
1\leq r\leq n,
$
is projective over $T(\Bbbk A_n)$, and therefore is $\tau$-rigid.

For $(
P_{b+1}\oplus I_a,
\left(\begin{smallmatrix}0&0\\
\nu(\partial)&0\end{smallmatrix}\right)
)
$, since
$
\tau(Y)=[a+1,b+1],
$
the projective cover of $\tau(Y)$ is $P_{a+1}$. The relevant
$\mathrm{Hom}$-vanishing condition becomes
$
\mathrm{Hom}_{\Bbbk A_n}(P_{a+1},P_{b+1})=0,
$
which, by \eqref{eq:interval hom}, is equivalent to $a<b$. The other $\mathrm{Hom}$-vanishing condition is automatic, since
$
\mathrm{Hom}_{\Bbbk A_n}(P_{b+2},I_a)=0
$
whenever $b\leq n-2$. Consequently, the indecomposable $\tau$-rigid minimal syzygy modules
are
$
\Omega Z([a,b]),
\
1\leq a<b<n.
$
Thus the indecomposable $\tau$-rigid $T(A)$-modules are
\begin{align*}
X_{a,b}&:=Z([a,b]), \ 1\leq a\leq b<n,\\
R_r&:=Z(P_r),\ 2\leq r\leq n,\\
L_s&:=L(P_s), \ 1\leq s\leq n,\\
\Omega_{a,b}&:=\Omega Z([a,b]),\ 1\leq a<b<n.
\end{align*}
Their number is
$$
\frac{n(n-1)}2+(n-1)+n+\frac{(n-1)(n-2)}2=n^2.
$$

Let $M$ be a basic $T(\Bbbk A_n)$-module of the form
$$M=\bigoplus_{[a,b]\in \mathcal{X}}X_{a,b}\oplus \bigoplus_{r\in \mathcal{R}}R_r\oplus \bigoplus_{s\in \mathcal{S}}L_s\oplus \bigoplus_{[a,b]\in \mathcal{Y}}\Omega_{a,b},$$
where the indexing sets contain no repetitions. Using (\ref{eq:interval hom}), Theorem~\ref{thm:main result}(a) says that $M$ is $\tau$-rigid exactly when the following interval $\mathrm{Hom}$-spaces vanish:
\begin{align*}&\mathrm{Hom}_{\Bbbk A_n}(\bigoplus_{[a,b]\in \mathcal{X}}[a,b]\oplus \bigoplus_{r\in \mathcal{R}}P_r\oplus \bigoplus_{s\in \mathcal{S}}P_s\oplus \bigoplus_{[a,b]\in \mathcal{Y}}P_{b+1}, \bigoplus_{[c,d]\in \mathcal{X}}[c+1,d+1])=0,\\
&\mathrm{Hom}_{\Bbbk A_n}(\bigoplus_{[a,b]\in \mathcal{Y}}[a,b], \bigoplus_{[c,d]\in \mathcal{Y}}[c+1,d+1]\oplus \bigoplus_{r\in \mathcal{R}}I_r)=0\\
&\mathrm{Hom}_{\Bbbk A_n}(P_1^{|\mathcal{R}|}\oplus \bigoplus_{[a,b]\in \mathcal{Y}}P_{a+1}, \bigoplus_{[c,d]\in \mathcal{X}}[c,d]\oplus \bigoplus_{r\in \mathcal{R}}P_r\oplus \bigoplus_{s\in \mathcal{S}}P_s\oplus \bigoplus_{[a,b]\in \mathcal{Y}}P_{b+1})=0\\
&\mathrm{Hom}_{\Bbbk A_n}(\bigoplus_{r\in \mathcal{R}}P_{r+1}\oplus \bigoplus_{[a,b]\in \mathcal{Y}}P_{b+2}, \bigoplus_{s\in \mathcal{S}}I_s\oplus \bigoplus_{[a,b]\in \mathcal{Y}}I_a)=0.
\end{align*}

 Using (\ref{eq:interval hom}), one has  
$$\operatorname{supp}U(X_{a,b})=[a,b], \ \operatorname{supp}U(R_r)=[r,n], \ \operatorname{supp}U(L_s)=[1,n]$$ and $$ \operatorname{supp}U(\Omega_{a,b})=[1,a]\cup [b+1,n].$$
Therefore, we have
$$\operatorname{supp}U(M)=\bigcup_{[a,b]\in \mathcal{X}}[a,b]\cup \bigcup_{r\in \mathcal{R}}[r,n]\cup \bigcup_{s\in \mathcal{S}}[1,n]\cup \bigcup_{[c,d]\in \mathcal{Y}}([1,c]\cup [d+1,n]).$$
A useful immediate consequence is:
$$\mathcal{S}\neq\varnothing\Longrightarrow J(M)=\varnothing.$$
Therefore, any support $\tau$-tilting module containing a summand $L(P_s)$ must have $|M|=n$, so it is actually a $\tau$-tilting module. Later we will see that the converse is also true, i.e., every basic $\tau$-tilting $T(\Bbbk A_n)$-module contains
a direct summand of the form $L(P_i)$.

Another useful consequence is:
\begin{align*} (X_{a,b},L(P_j))\text{ is a $\tau$-rigid pair} &\Longleftrightarrow j<a\text{ or }j>b,\\
(R_r,L(P_j))\text{ is a $\tau$-rigid pair} &\Longleftrightarrow j<r,\\
 (\Omega_{a,b},L(P_j))\text{ is a $\tau$-rigid pair} &\Longleftrightarrow
a<j\leq b,
\end{align*}
and $(L_s,L(P_j))$ is never a $\tau$-rigid pair.

\subsection{The rank-two basic $\tau$-rigid $T(\Bbbk A_n)$-modules} From the $\mathrm{Hom}$-vanishing conditions of Theorem \ref{thm:main result}, we can determine the rank-two basic $\tau$-rigid $T(\Bbbk A_n)$-modules. 

For two modules of type $X_{a,b}$. The module $X_{a,b}\oplus X_{c,d}=Z([a,b])\oplus Z([c,d])$ with $a\leq b<n, c\leq d<n$ is $\tau$-rigid if and only if 
$[a,b]\oplus [c,d]$
is a $\tau$-rigid $\Bbbk A_n$-module. Equivalently,
$$\mathrm{Hom}_{\Bbbk A_n}([a,b],\tau[c,d])=0\quad \text{and} \quad \mathrm{Hom}_{\Bbbk A_n}([c,d],\tau[a,b])=0.$$
Now $\tau[c,d]=[c+1,d+1]$ shows that 
$$\mathrm{Hom}_{\Bbbk A_n}([a,b],\tau[c,d])\neq 0\Longleftrightarrow c+1\leq a\leq d+1\leq b.$$
Similarly, $$\mathrm{Hom}_{\Bbbk A_n}([c,d],\tau[a,b])\neq 0\Longleftrightarrow a+1\leq c\leq b+1\leq d.$$
Interchange the two summands, when necessary, so that 
$$a<c, \quad \text{or}\quad a=c \quad \text{and} \quad b>d.$$
With this convention, $c+1\leq a$ is impossible, so the first $\mathrm{Hom}$-space always vanishes. Thus $X_{a,b}\oplus X_{c,d}$ is $\tau$-rigid if and only if 
$a+1\leq c\leq b+1\leq d$
does not hold. This is equivalent to 
$$a\leq c\leq d\leq b\quad \text{or} \quad b+2\leq c\leq d.$$
Notice that $b,d<n$, so for an interval $[a,b]$, write $\ell=b-a+1$, there are $n-\ell$ intervals with length $\ell$, and each contains 
$\frac{\ell(\ell+1)}{2}-1$
proper subintervals. Hence the number of the pairs $[c,d]\subsetneq [a,b]$ is 
$$\sum_{\ell=1}^{n-1}(n-\ell)(\frac{\ell(\ell+1)}{2}-1)=\frac{n(n-1)(n-2)(n+5)}{24}.$$
On the other hand, fix the right endpoint $b$ of the left interval. There are $b$ choices for its left endpoint $a$. The number of intervals $[c,d]$ satisfying $c\geq b+2$ is
$$\sum_{c=b+2}^{n-1}(n-c)=\binom{n-b-1}{2}.$$
Therefore there are $$\sum_{b=1}^{n-2}b\binom{n-b-1}{2}=\frac{n(n-1)(n-2)(n-3)}{24}$$
choices of the intervals $[a,b], [c,d]$ with $c\geq b+2$. Consequently, the number of the modules $X_{a,b}\oplus X_{c,d}$ up to isomorphism is
$$\frac{(n+5)n(n-1)(n-2)}{24}+\frac{n(n-1)(n-2)(n-3)}{24}=\frac{n(n-1)(n-2)(n+1)}{12}.$$

For the module of the form $R_r\oplus L_s=Z(P_r)\oplus L(P_s)$, the minimal projective resolution of $\nu(P_r)=I_r$ is
$0\to P_{r+1}\to P_1\to I_r\to 0$
with $P_{n+1}=0$. Condition (3) of Theorem~\ref{thm:main result} becomes
$$\mathrm{Hom}_{\Bbbk A_n}(P_1,P_s)=0\quad \mbox{and}\quad \mathrm{Hom}_{\Bbbk A_n}(P_{r+1},I_s)=0.$$
The first condition is equivalent to $s\geq 2$, while the second one is equivalent to $s\leq r$. Thus $R_r\oplus L_s$ is $\tau$-rigid if and only if $2\leq s\leq r$. This gives 
$$\sum_{r=2}^n(r-1)=\frac{n(n-1)}{2}$$
such modules.

Similarly, we can deduce that every basic rank-two $\tau$-rigid $T(\Bbbk A_n)$-module belongs to one of the following ten types.
\begin{tiny}\begin{table}[H]
    \centering
    \caption{rank-two basic $\tau$-rigid $T(\Bbbk A_n)$-modules}\label{tab:ranktwo}
      \renewcommand{\arraystretch}{1.2}
      \setlength{\extrarowheight}{1pt}
    \begin{tabularx}{0.9\textwidth}{>{\centering\arraybackslash}X|>{\centering\arraybackslash}X|c}
        pair & compatibility condition & number \\
        \hline
        $X_{a,b}\oplus X_{c,d}
    (a<c, \text{or}\ a=c\ \text{and} \ b>d)$ & $[c,d]\subsetneq[a,b]$, or $c\geq b+2$        & $\frac{n(n-1)(n-2)(n+1)}{12}$       \\
       \hline $X_{a,b}\oplus R_r$ & $a\geq 2$ and ($r\geq b+2\ \text{or}\ r\leq a)$ & $\frac{(n-1)(n-2)(2n-3)}{6}$ \\
        \hline $X_{a,b}\oplus L_s$ & $s\leq a$ or $s\geq b+2$ & $\frac{n(n-1)(2n-1)}{6}$ \\
         \hline $X_{a,b}\oplus \Omega_{c,d}$ & $b\leq c$, or $a\geq d+1$, or $c+2\leq a\leq b\leq d-1$ & $\frac{(n-1)(n-2)(n^2-3n+4)}{8}$ \\
         \hline $R_r\oplus R_s$ & $r\neq s$ & $\binom{n-1}{2}$ \\
         \hline $R_r\oplus L_s$ & $2\leq s \leq r\leq n$ & $\frac{n(n-1)}{2}$ \\
         \hline $R_r\oplus \Omega_{c,d}$ & $r\geq d+1$ & $\frac{n(n-1)(n-2)}{6}$ \\
         \hline $L_r\oplus L_s$ & $r\neq s$ & $\binom{n}{2}$\\
         \hline $L_s\oplus \Omega_{c,d}$ & $c+2\leq s \leq d+1$ & $\frac{n(n-1)(n-2)}{6}$\\
          \hline $\Omega_{a,b}\oplus \Omega_{c,d}
    (a<c, \mbox{or}\ a=c\ \text{and} \ b>d)$ & $[c,d]\subsetneq [a,b]$ & $\frac{(n-1)(n-2)(n-3)(n+4)}{24}$\\
    \hline
    \end{tabularx}
\end{table}
\end{tiny}
Since the $\tau$-rigid complex is flag, the edge information then determines all higher-rank $\tau$-rigid modules as cliques. So classifying rank-two $\tau$-rigid modules is not merely the first nontrivial case, it determines the whole simplicial complex.

By adding the ten contributions, we obtain a formula for the number of basic rank-two $\tau$-rigid $T(\Bbbk A_n)$-modules $\frac{n^2(n^2-1)}{4}=\binom{n}{2}\binom{n+1}{2}.$

In subsection 4.2, we have determined $\tau$-rigid pairs $(M,L(P))$ satisfying $|M|=|P|=1$. Combining with Table~\ref{tab:ranktwo}, we see that the basic $\tau$-rigid pairs $(M,L(P))$ satisfying $|M|+|P|=2$ are precisely the pairs belonging to the following three classes:

\medskip \noindent \textup{(I)} Pairs of the form $(M,0)$, where $M$ is one of the rank-two basic $\tau$-rigid modules listed in Table~\ref{tab:ranktwo}.

\medskip \noindent \textup{(II)} Pairs of the form $(N,L(P_j))$, where $N$ is indecomposable and one of the following conditions holds:
\[
\begin{array}{c|c}
    N & \text{condition on } j \\ \hline
    X_{a,b} & j<a \text{ or } j>b \\
    R_r & j<r \\
     \Omega_{a,b} & a<j\leq b.
\end{array}
\]

\medskip \noindent \textup{(III)} Pairs of the form $(0,L(P_i\oplus P_j)), 1\leq i<j\leq n.$

Therefore, we can deduce that the formula for the number of basic $\tau$-rigid $T(\Bbbk A_n)$-pairs satisfying $|M|+|P|=2$
$$\binom{n}{2}\binom{n+1}{2} +\frac{n^2(n-1)}{2}+\binom{n}{2}=\binom{n}{2}\binom{n+2}{2},$$
which agrees exactly with the known $f_1$-entry for the two-term pretilting complex of a Brauer tree algebra in \cite{AMN}. Recall that, a basic $\tau$-rigid pair of $T(\Bbbk A_n)$ satisfying $|M|+|P|=2$ corresponds to a basic two-term pretilting complex with two indecomposable summands \cite{AIR} and $T(\Bbbk A_n)$ is a Brauer tree algebra \cite{WZ}. 

\begin{exm} Take $B=T(\Bbbk A_3)$, then $G_B$ has nine vertices:
$$X_{1,1}< X_{2,2}< X_{1,2}<R_2<R_3< L_1< L_2< L_3< \Omega_{1,2}.$$
By table \ref{tab:ranktwo}, the neighbors of each vertex are:
\begin{tiny}\begin{table}[H]
    \centering
    \caption{Edges of $G_{T(\Bbbk A_3)}$}\label{tab:mmedge}
      \renewcommand{\arraystretch}{1.2}
      \setlength{\extrarowheight}{-.1pt}
    \begin{tabular}{c|c}
        vertex & compatible vertices  \\
        \hline
        $X_{1,1}$ & $X_{1,2}, L_1, L_3, \Omega_{1,2}$        \\
        $X_{2,2}$ & $X_{1,2}, R_2, L_1, L_2$ \\
         $X_{1,2}$ & $L_1$  \\
         $R_2$ & $R_3,L_2$  \\
         $R_3$ & $L_2,L_3,\Omega_{1,2}$  \\
         $L_1$ & $L_2,L_3$  \\
         $L2$ & $L_3$  \\
         $L_3$ & $\Omega_{1,2}$  \\
         $\Omega_{1,2}$ & $\varnothing $  \\
    \end{tabular}
\end{table}
\end{tiny}
Thus $T(\Bbbk A_3)$ has $18$ rank-two basic $\tau$-rigid modules up to isomorphism.  There are $10$ cliques of rank-three in $G_{T(\Bbbk A_3)}$ and thus $T(\Bbbk A_3)$ has $10$ rank-three basic $\tau$-rigid modules up to isomorphism: the nonempty intersections of compatible vertices are as follows:
\begin{tiny}\begin{table}[H]
    \centering
    \caption{Triangles of $G_{T(\Bbbk A_3)}$}\label{tab:G_Btriangles}
      \renewcommand{\arraystretch}{1.2}
      \setlength{\extrarowheight}{-.1pt}
    \begin{tabular}{c|c}
        nonempty intersection & corresponding basic $\tau$-rigid modules  \\
        \hline
        $N(X_{1,1})\cap N(X_{1,2})=\{L_1\}$ & $Z(S_1)\oplus Z(I_2)\oplus L(P_1)$        \\
        $N(X_{1,1})\cap N(L_1)=\{L_3\}$ & $Z(S_1)\oplus L(P_1)\oplus L(P_3)$ \\
         $N(X_{1,1})\cap N(L_3)=\{\Omega_{1,2}\}$ & $Z(S_1)\oplus L(P_3)\oplus \Omega Z(I_2)$  \\
         $N(X_{2,2})\cap N(X_{1,2})=\{L_1\}$ & $Z(S_2)\oplus Z(I_2)\oplus L(P_1)$  \\
         $N(X_{2,2})\cap N(R_2)=\{L_2\}$ & $Z(S_2)\oplus Z(P_2)\oplus L(P_2)$  \\
         $N(X_{2,2})\cap N(L_1)=\{L_2\}$ & $Z(S_2)\oplus L(P_1)\oplus L(P_2)$  \\
         $N(R_2)\cap N(R_3)=\{L_2\}$ & $Z(P_2)\oplus Z(P_3)\oplus L(P_2)$  \\
         $N(R_3)\cap N(L_2)=\{L_3\}$ & $Z(P_3)\oplus L(P_2)\oplus L(P_3)$  \\
         $N(R_3)\cap N(L_3)=\{\Omega_{1,2}\}$ & $Z(P_3)\oplus L(P_3)\oplus \Omega Z(I_2)$  \\
         $N(L_1)\cap N(L_2)=\{L_3\}$ & $L(P_1)\oplus L(P_2)\oplus L(P_3)$
    \end{tabular}
\end{table}
\end{tiny}
Since $T(\Bbbk A_3)$ has three simple modules, a basic $\tau$-rigid module has at most three indecomposable direct summands. Consequently, $G_{T(\Bbbk A_3)}$ has no clique of size four, and thus $T(\Bbbk A_3)$ has $1+9+18+10=28$
basic $\tau$-rigid modules up to isomorphism. 

For constructing the support $\tau$-tilting pairs of $T(\Bbbk A_3)$, we define the
\emph{support $\tau$-tilting compatibility graph} $G_B^{\rm s\tau}$ to be the $1$-skeleton of the support $\tau$-tilting complex $\Delta_{\mathrm{s}\tau}(B)$. 
In the case $B=T(\Bbbk A_3)$, the vertex set of $G_B^{\rm s\tau}$ is
$$X_{1,1}< X_{2,2}< X_{1,2}< R_2<R_3<L_1< L_2<L_3<\Omega_{1,2}< L_1[1]<L_2[1]< L_3[1].$$
For a module vertex $M$, Theorem~\ref{thm:main result}(b) gives
$$M\mathrel{\relbar} L_j[1]\Longleftrightarrow\mathrm{Hom}_{T(\Bbbk A_3)}(L(P_j),M)=0\Longleftrightarrow j\in J(M).$$
Combining this with Table~\ref{tab:mmedge}, we obtain
\begin{tiny}\[
\begin{array}{c|c|c|c}
    M & \operatorname{supp}U(M)& \text{adjacent shifted vertices}&\text{compatible module vertices}\\
     \hline
    X_{1,1} & \{1\} & L_2[1], L_3[1] & X_{1,2}, L_1, L_3, \Omega_{1,2} \\
    X_{2,2} & \{2\} & L_1[1], L_3[1]& X_{1,2}, R_2, L_1, L_2\\
     X_{1,2} & \{1,2\}& L_3[1]& L_1\\
     R_2&\{2,3\}& L_1[1]& R_3,L_2\\
     R_3&\{3\} & L_1[1],L_2[1]& L_2,L_3,\Omega_{1,2}\\
     L_1 &\{1,2,3\} & \varnothing& L_2,L_3\\
      L_2 &\{1,2,3\} & \varnothing& L_3\\
       L_3 &\{1,2,3\} & \varnothing& \Omega_{1,2}\\
        \Omega_{1,2} &\{1,3\} & L_2[1]& \varnothing
\end{array}
\]\end{tiny}
In addition $$L_1[1]\mathrel{\relbar} L_2[1], \quad L_1[1]\mathrel{\relbar} L_3[1], \quad L_2[1]\mathrel{\relbar} L_3[1]$$
are all edges. Therefore $G_{T(\Bbbk A_3)}^{\rm s\tau}$ has $30$ edges.

Since $T(\Bbbk A_3)$ has three simple modules, the basic support $\tau$-tilting pairs are exactly the triangles of the extended compatibility graph $G_{T(\Bbbk A_3)}^{\rm s\tau}$ listed as follows:
\begin{tiny}\begin{table}[H]
    \centering
      \renewcommand{\arraystretch}{1.2}
      \setlength{\extrarowheight}{-.1pt}
    \begin{tabular}{c|c}
        nonempty intersection & corresponding basic support $\tau$-tilting pairs  \\
        \hline
        $N(X_{1,1})\cap N(X_{1,2})=\{L_1, L_3[1]\}$ & $(Z(S_1)\oplus Z(I_2)\oplus L(P_1),0), (Z(S_1)\oplus Z(I_2), L(P_3))$        \\
        $N(X_{1,1})\cap N(L_1)=\{L_3\}$ & $(Z(S_1)\oplus L(P_1)\oplus L(P_3),0)$ \\
         $N(X_{1,1})\cap N(L_3)=\{\Omega_{1,2}\}$ & $(Z(S_1)\oplus L(P_3)\oplus \Omega Z(I_2),0)$  \\
          $N(X_{1,1})\cap N(\Omega_{1,2})=\{L_2[1]\}$ & $((Z(S_1)\oplus \Omega Z(I_2),L(P_2))$  \\
           $N(X_{1,1})\cap N(L_2[1])=\{L_3[1]\}$ & $(Z(S_1),L(P_2)\oplus L(P_3))$  \\
         $N(X_{2,2})\cap N(X_{1,2})=\{L_1,L_3[1]\}$ & $(Z(S_2)\oplus Z(I_2)\oplus L(P_1),0), (Z(S_2)\oplus Z(I_2), L(P_3))$  \\
         $N(X_{2,2})\cap N(R_2)=\{L_2,L_1[1]\}$ & $(Z(S_2)\oplus Z(P_2)\oplus L(P_2),0), (Z(S_2)\oplus Z(P_2), L(P_1))$  \\
         $N(X_{2,2})\cap N(L_1)=\{L_2\}$ & $(Z(S_2)\oplus L(P_1)\oplus L(P_2),0)$  \\
         $N(X_{2,2})\cap N(L_1[1])=\{L_3[1]\}$ & $(Z(S_2), L(P_1)\oplus L(P_3))$  \\
         $N(R_2)\cap N(R_3)=\{L_2,L_1[1]\}$ & $(Z(P_2)\oplus Z(P_3)\oplus L(P_2),0), (Z(P_2)\oplus Z(P_3), L(P_1))$  \\
         $N(R_3)\cap N(L_2)=\{L_3\}$ & $(Z(P_3)\oplus L(P_2)\oplus L(P_3),0)$  \\
         $N(R_3)\cap N(L_3)=\{\Omega_{1,2}\}$ & $(Z(P_3)\oplus L(P_3)\oplus \Omega Z(I_2),0)$  \\
         $N(R_3)\cap N(\Omega_{1,2})=\{L_2[1]\}$ & $(Z(P_3)\oplus \Omega Z(I_2),L(P_2))$  \\
         $N(R_3)\cap N(L_1[1])=\{L_2[1]\}$ & $(Z(P_3),L(P_1)\oplus L(P_2))$  \\
         $N(L_1)\cap N(L_2)=\{L_3\}$ & $(L(P_1)\oplus L(P_2)\oplus L(P_3),0)$\\
         $N(L_1[1])\cap N(L_2[1])=\{L_3[1]\}$ & $(0,L(P_1)\oplus L(P_2)\oplus L(P_3))$
    \end{tabular}
\end{table}
\end{tiny}
Therefore, there are $20$ basic support $\tau$-tilting pairs in $T(\Bbbk A_3)\text{-}\mathrm{mod}$.
\end{exm}

\subsection{Projective-free basic $\tau$-rigid $T(\Bbbk A_n)$-modules} 
The following result does not hold for arbitrary finite-dimensional algebras: there exist basic $\tau$-tilting modules having no projective direct summands \cite{XieZhang}. Thus the corollary reflects a special feature of the trivial extension of the linearly oriented quiver of type $A_n$.

\begin{cor}\label{cor:without projective summands}
Let $M$ be a basic $\tau$-rigid $T(\Bbbk A_n)$-module. If
$
L(P_i)\notin\operatorname{add}M$ for every $
1\leq i\leq n,
$
then
$
|M|\leq n-1.
$
Consequently, every basic $\tau$-tilting $T(\Bbbk A_n)$-module contains
a direct summand of the form $L(P_i)$.
\end{cor}
\begin{proof}
Write
$$
M=
\bigoplus_{[a,b]\in\mathcal X}Z([a,b])
\oplus
\bigoplus_{r\in\mathcal R}Z(P_r)
\oplus
\bigoplus_{[c,d]\in\mathcal Y}\Omega Z([c,d]),
$$
where
$
1\leq a\leq b<n,\qquad
2\leq r\leq n,\qquad
1\leq c<d<n.
$
There is no summand of the form $L(P_i)$ by assumption. We prove that
$$
\operatorname{supp}U(M)\neq{1,\ldots,n}.
$$
Equivalently, there exists $j$ such that
$
\mathrm{Hom}_{\Bbbk A_n}(P_j,U(M))=0.
$

We distinguish three cases.

\medskip
\noindent
Case 1: $\mathcal{Y}\neq\varnothing$.

The compatibility condition for two minimal syzygies implies that
the intervals indexing the elements of $\mathcal{Y}$ are totally
ordered by strict inclusion. Choose an interval $
[c,d]\in\mathcal{Y}
$
which is minimal with respect to inclusion. The support of the corresponding term in $U(M)$ is 
$\operatorname{supp}(P_{d+1}\oplus I_c)=[1,c]\cup [d+1,n]$
and hence it does not contain the vertex $c+1$. Every other interval $[c',d']\in\mathcal Y$ contains $[c,d]$.
Therefore $
c'\leq c< c+1\leq d\leq d',
$
so $c+1$ is also absent from $
\operatorname{supp}(P_{d'+1}\oplus I_{c'}).
$

Now let $Z([a,b])$ be a summand of $M$. The compatibility
condition between $Z([a,b])$ and $\Omega Z([c,d])$ is
$$
b\leq c,
\qquad\text{or}\qquad
a\geq d+1,
\qquad\text{or}\qquad
c+2\leq a\leq b\leq d-1.
$$
In each of these three cases,
$
c+1\notin[a,b].
$

Similarly, the compatibility condition between $Z(P_r)$ and
$\Omega Z([c,d])$ is
$
r\geq d+1.
$
Thus
$
c+1\notin\operatorname{supp}P_r=[r,n].
$
It follows that
$
c+1\notin\operatorname{supp}U(M).
$

\medskip
\noindent
Case 2: $\mathcal{Y}=\varnothing$ and
$\mathcal{R}\neq\varnothing$.

For every $Z([a,b])$ occurring in $M$ and every
$Z(P_r)$ occurring in $M$, their compatibility condition gives
$
a\geq2.
$
Furthermore,
$
r\geq2
$
for every $r\in\mathcal{R}$. Hence none of the summands of $U(M)$
has vertex $1$ in its support. Therefore
$
1\notin\operatorname{supp}U(M).
$

\medskip
\noindent
Case 3: $\mathcal{Y}=\varnothing=\mathcal{R}$.

In this case, $
U(M)=\bigoplus_{[a,b]\in\mathcal{X}}[a,b].
$
Since every interval satisfies $b<n$, none of them contains the
vertex $n$. Thus
$
n\notin\operatorname{supp}U(M).
$

In all cases, there exists a vertex
$
j\in{1,\ldots,n}
$ such that
$
j\notin\operatorname{supp}U(M),
$
which is equivalent to 
$
\mathrm{Hom}_{\Bbbk A_n}(P_j,U(M))=0.
$
By Theorem~\ref{thm:main result}, it follows that
$
(M,L(P_j))
$
is a basic $\tau$-rigid pair. Every basic $\tau$-rigid pair over
an algebra with $n$ simple modules has at most $n$ indecomposable
summands. Hence
$
|M|+1\leq n,
$
and therefore
$
|M|\leq n-1.
$

Finally, if $M$ is a basic $\tau$-tilting $T(\Bbbk A_n)$-module, then
$
|M|=n.
$
The preceding inequality shows that $M$ cannot avoid all the
modules $L(P_i)$. Consequently,
$
L(P_i)\in\operatorname{add}M
$
for at least one $i$.
\end{proof}

\begin{rem} There is a simplicial interpretation of Corollary \ref{cor:without projective summands}: the induced subcomplex obtained by deleting all vertices $\{L(P_1),\cdots, L(P_n)\}$ has dimension at most $n-2$.
\end{rem}

\section{Examples of type $D_4$}

In this section we use Theorem~\ref{thm:main result} and the compatibility graph to determine the basic $\tau$-rigid modules of trivial extensions of quivers of $D_4$ type.

\begin{exm} \label{exm:ssD_4}Let $Q_{\rm ss}$ be the following quiver of type $D_4$
$$
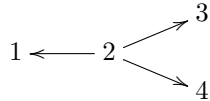
{\xymatrix@R=0.5ex{
&&3 \\
1 \ar@{<-}[r]  &{2}\ar@{->}[ur]
\ar@{->}[dr]\\
&&4 }}$$	
 We denote $M_d$ the unique indecomposable $\Bbbk Q_{\rm ss}$-module with dimension vector 
$d=(d_1,d_2,d_3,d_4).$
Then 
\begin{align*}P_1&=S_1=M_{1000}, \ P_2=M_{1111}, \ P_3=S_3=M_{0010}, \ P_4=S_4=M_{0001},\\
I_1&=M_{1100}, \ I_2=S_2=M_{0100}, \ I_3=M_{0110}, \ I_4=M_{0101}.
\end{align*}
Write
$$X_d=Z(M_d), \quad \Omega_d=\Omega Z(M_d), \quad R_i=Z(P_i), \quad L_i=L(P_i).$$
Then, by Lemma~\ref{lem:ind of tau-rigid}, there are twenty indecomposable $\tau$-rigid $T(\Bbbk Q_{\rm ss})$-modules 
\begin{align*}
&X_{0100}<X_{1100}< X_{0110}<X_{0101}< X_{1110}< X_{1101}< X_{0111}<X_{1211}<R_1< R_3\\
&< R_4<L_1< L_2< L_3< L_4<\Omega_{0100}<\Omega_{1100}<  \Omega_{0110}< \Omega_{0101}< \Omega_{1211}.
\end{align*} We extend this ordering the vertices of support $\tau$-tilting complex by requiring 
$$L_1[1]<L_2[1]<L_3[1]<L_4[1]$$ and by declaring every shifted projective vertex $L_i[1]$ to be larger than every module vertex. As in $A_3$ case, Theorem~\ref{thm:main result} determines the support $\tau$-tilting compatibility graph $G_{T(\Bbbk Q_{\rm ss})}^{\rm s\tau}$:
\begin{tiny}\[
\begin{array}{c|c|c|c}
    M & \operatorname{supp}U(M)& \text{adjacent shifted vertices}&\text{compatible module vertices}\\
     \hline
    X_{0100} & \{2\} & L_1[1], L_3[1], L_4[1] & X_{1100}, X_{0110}, X_{0101}\\
    X_{1100} & \{1,2\} & L_3[1], L_4[1]& X_{0110}, X_{0101},X_{1110}, X_{1101}, X_{1211}, L_1\\
     X_{0110} & \{2,3\}& L_1[1], L_4[1]& X_{0101}, X_{1110}, X_{0111}, X_{1211}, L_3\\
      X_{0101} & \{2,4\}& L_1[1], L_3[1]& X_{1101}, X_{0111}, X_{1211},L_4\\
       X_{1110} & \{1,2,3\}& L_4[1]& X_{1101}, X_{0111}, X_{1211},L_1,L_2,L_3\\
        X_{1101} & \{1,2,4\}& L_3[1]&  X_{0111}, X_{1211},L_1,L_2,L_4\\
         X_{0111} & \{2,3,4\}& L_1[1]&   X_{1211},L_2,L_3,L_4\\
          X_{1211} & \{1,2,3.4\}& \varnothing & \varnothing\\
     \hline R_1&\{1\}& L_2[1], L_3[1], L_4[1]& R_3,R_4,L_1,\Omega_{0100}, \Omega_{0110}, \Omega_{0101}\\
     R_3&\{3\} & L_1[1],L_2[1], L_4[1]& R_4,L_3,\Omega_{0100}, \Omega_{1100}, \Omega_{0101}\\
     R_4&\{4\} & L_1[1],L_2[1], L_3[1]&L_4,\Omega_{0100}, \Omega_{1100}, \Omega_{0110}\\
     \hline L_1 &\{1,2\} & L_3[1], L_4[1]& L_2,L_3, L_4, \Omega_{0110}, \Omega_{0101}, \Omega_{1211}\\
      L_2 &\{1,2,3,4\} & \varnothing& L_3, L_4\\
       L_3 &\{2,3\} & L_1[1], L_4[1]& L_4,\Omega_{1100}, \Omega_{0101}, \Omega_{1211}\\
       L_4 &\{2,4\} & L_1[1], L_3[1]& \Omega_{1100}, \Omega_{0110}, \Omega_{1211}\\
       \hline \Omega_{0100} &\{1,2,3,4\} & \varnothing &  \Omega_{1100}, \Omega_{0110}, \Omega_{0101}\\
       \Omega_{1100} &\{2,3,4\} & L_1[1] &  \Omega_{0110}, \Omega_{0101}, \Omega_{1211}\\
       \Omega_{0110} &\{1,2,4\} & L_3[1] &  \Omega_{0101}, \Omega_{1211}\\
       \Omega_{0101} &\{1,2,3\} &  L_4[1] & \Omega_{1211}\\
       \Omega_{1211} &\{1,2,3,4\} & \varnothing & \varnothing
\end{array}
\]\end{tiny}
In addition \begin{align*}&L_1[1]\mathrel{\relbar} L_2[1],  L_1[1]\mathrel{\relbar} L_3[1],  L_1[1]\mathrel{\relbar} L_4[1], \\
& L_2[1]\mathrel{\relbar} L_3[1],  L_2[1]\mathrel{\relbar} L_4[1],  L_3[1]\mathrel{\relbar} L_4[1]\end{align*}
are all edges. Therefore, $G_{T(\Bbbk Q_{\rm ss})}^{\rm s\tau}$ has $108$ edges and $T(\Bbbk Q_{\rm ss})$ has $72$ basic rank-two $\tau$-rigid modules.

Using these edges and the intersections of neighbours, a direct clique enumeration gives $168$ triangles. They split according to the number of shifted projective vertices as
$$(t_{3,0},t_{2,1}, t_{1,2}, t_{0,3})=(92,54,18,4),$$
where $t_{r,s}$ denotes the number of vertices of a triangle with $r$ module vertices and $s$ shifted projective vertices. 

The $168$ triangles are precisely the basic rank-three $\tau$-rigid pairs, i.e. the almost complete support $\tau$-tilting pairs. Each has exactly two complements. Therefore the total triangle-four-clique incidence number is $2\cdot168$. Every four-clique contains exactly four triangles, so there are exactly $\frac{2\cdot168}{4}=84$ four--cliques. They split according to the number of shifted projective vertices as
$$(q_{4,0}, q_{3,1}, q_{2,2}, q_{1,3}, q_{0,4})=(39,28,12,4,1),$$
where $q_{r,s}$ denotes the number of support $\tau$-tilting pairs $(M,L(P))$ with $|M|=r$ and $|P|=s$.
\end{exm}

\begin{exm} \label{exm:mixD_4} Let $Q_{\rm mix}$ be the following quiver of type $D_4$
$$
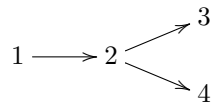
{\xymatrix@R=0.5ex{
&&3 \\
1 \ar@{->}[r]  &{2}\ar@{->}[ur]
\ar@{->}[dr]\\
&&4 }}$$	
We denote $M_d$ the unique indecomposable $\Bbbk Q_{\rm mix}$-module with dimension vector 
$d=(d_1,d_2,d_3,d_4).$
Then 
\begin{align*}P_1&=M_{1111}, \ P_2=M_{0111}, \ P_3=S_3=M_{0010}, \ P_4=S_4=M_{0001},\\
I_1&=S_1=M_{1000}, \ I_2=M_{1100}, \ I_3=M_{1110}, \ I_4=M_{1101}.
\end{align*}
Write
$$X_d=Z(M_d), \quad \Omega_d=\Omega Z(M_d), \quad R_i=Z(P_i), \quad L_i=L(P_i).$$
Then, by Lemma~\ref{lem:ind of tau-rigid}, there are twenty indecomposable $\tau$-rigid $T(\Bbbk Q_{\rm mix})$-modules 
\begin{align*}
&X_{1000}<X_{0100}< X_{1100}<X_{0110}< X_{0101}< X_{1110}< X_{1101}<X_{1211}<R_2< R_3\\
&< R_4<L_1< L_2< L_3< L_4<\Omega_{0100}<\Omega_{1100}<  \Omega_{1110}< \Omega_{1101}< \Omega_{1211}.
\end{align*} We extend this ordering the vertices of support $\tau$-tilting complex by requiring 
$$L_1[1]<L_2[1]<L_3[1]<L_4[1]$$ and by declaring every shifted projective vertex $L_i[1]$ to be larger than every module vertex. As in $A_3$ case, Theorem~\ref{thm:main result} determines the support $\tau$-tilting compatibility graph $G_{T(\Bbbk Q_{\rm mix})}^{\rm s\tau}$:
\begin{tiny}\[
\begin{array}{c|c|c|c}
    M & \operatorname{supp}U(M)& \text{adjacent shifted vertices}&\text{compatible module vertices}\\
     \hline
    X_{1000} & \{1\} & L_2[1], L_3[1], L_4[1] & X_{1100}, X_{1110}, X_{1101},L_1,L_3,L_4,\Omega_{1110}, \Omega_{1101}, \Omega_{1211}\\
    X_{0100} & \{2\} & L_1[1], L_3[1],L_4[1]& X_{1100}, X_{0110},X_{0101}, X_{1110}, X_{1101}, X_{1211}\\
     X_{1100} & \{1,2\}& L_3[1], L_4[1]& X_{1110}, X_{1101}\\
      X_{0110} & \{2,3\}& L_1[1], L_4[1]& X_{0101}, X_{1110}, X_{1211},R_2,R_3,L_1,L_2,L_3\\
       X_{0101} & \{2,4\}& L_1[1], L_3[1]& X_{1101}, X_{1211},R_2, R_4, L_1,L_2,L_4\\
        X_{1110} & \{1,2,3\}& L_4[1]&  X_{1101}, X_{1211},L_1,L_3\\
         X_{1101} & \{1,2,4\}& L_3[1]&   X_{1211},L_1, L_4\\
          X_{1211} & \{1,2,3.4\}& \varnothing & L_1\\
     \hline R_2&\{2,3,4\}& L_1[1]& R_3,R_4,L_2\\
     R_3&\{3\} & L_1[1],L_2[1], L_4[1]& R_4,L_2, L_3,\Omega_{0100}, \Omega_{1100}, \Omega_{1101}\\
     R_4&\{4\} & L_1[1],L_2[1], L_3[1]&L_2,L_4,\Omega_{0100}, \Omega_{1100}, \Omega_{1110}\\
     \hline L_1 &\{1,2,3,4\} & \varnothing& L_2,L_3, L_4\\
      L_2 &\{1,2,3,4\} & \varnothing& L_3, L_4, \Omega_{0100}\\
       L_3 &\{1,2,3\} &L_4[1]& L_4,\Omega_{0100}, \Omega_{1101}, \Omega_{1211}\\
       L_4 &\{1,2,4\} & L_3[1]& \Omega_{0100}, \Omega_{1110}, \Omega_{1211}\\
       \hline \Omega_{0100} &\{1,2,3,4\} & \varnothing &  \Omega_{1100}, \Omega_{1110}, \Omega_{1101},\Omega_{1211}\\
       \Omega_{1100} &\{1,3,4\} & L_2[1] &  \Omega_{1110}, \Omega_{1101}\\
       \Omega_{1110} &\{1,4\} &L_2[1], L_3[1] &  \Omega_{1101}, \Omega_{1211}\\
       \Omega_{1101} &\{1,3\} & L_2[1], L_4[1] & \Omega_{1211}\\
       \Omega_{1211} &\{1,2,3,4\} & \varnothing & \varnothing
\end{array}
\]\end{tiny}
In addition \begin{align*}&L_1[1]\mathrel{\relbar} L_2[1],  L_1[1]\mathrel{\relbar} L_3[1],  L_1[1]\mathrel{\relbar} L_4[1], \\
& L_2[1]\mathrel{\relbar} L_3[1],  L_2[1]\mathrel{\relbar} L_4[1],  L_3[1]\mathrel{\relbar} L_4[1]\end{align*}
are all edges. Therefore, $G_{T(\Bbbk Q_{\rm mix})}^{\rm s\tau}$ has $110$ edges and $T(\Bbbk Q_{\rm mix})$ has $76$ basic rank-two $\tau$-rigid modules.

Using these edges and the intersections of neighbours, a direct clique enumeration gives $172$ triangles. They split according to the number of shifted projective vertices as
$$(t_{3,0},t_{2,1}, t_{1,2}, t_{0,3})=(101,50,17,4).$$
There are exactly $86$ four-cliques. They split according to the number of shifted projective vertices as
$$(q_{4,0}, q_{3,1}, q_{2,2}, q_{1,3},q_{0,4})=(44,26,11,4,1).$$
\end{exm}

The total numbers of basic support $\tau$-tilting modules in
Examples~\ref{exm:ssD_4} and~\ref{exm:mixD_4} are $84$ and $86$, respectively. These numbers
were already observed in \cite{WZ}, where they provide a negative
answer in type $D$ to the orientation-independence question considered
there. We also note that, for the source-sink orientation $Q_{\rm ss}$,
the number $84$ follows from the formula
\[
6\cdot 4^{\,n-2}
-
2\binom{2n-4}{n-2}
\]
in \cite{AA,BR}, since $T(\Bbbk Q_{\rm ss})$ is a symmetric algebra
with radical cube zero.

The compatibility graphs computed above contain finer information than
these total numbers. Indeed, their clique distributions determine the
numbers of basic $\tau$-rigid pairs according to the numbers of module
and shifted-projective summands. We encode these face distributions by
an $F$-triangle, following the terminology of Chapoton \cite{Chapoton}.
For a quiver $Q$, define
\[
F_Q(x,y)
=
\sum_{r,s\geq 0} f_{r,s}(Q)x^r y^s,
\]
where $f_{r,s}(Q)$ denotes the number of basic $\tau$-rigid pairs
$(M,L(P))$ over $T(\Bbbk Q)$ satisfying $
|M|=r, |P|=s.$
Thus the exponent of $x$ records the number of module vertices and the
exponent of $y$ records the number of shifted-projective vertices.
In particular, the coefficients of $F_Q(x,y)$ are precisely the
bigraded face numbers obtained from the support $\tau$-tilting
compatibility graph.

For the source--sink orientation $Q_{\rm ss}$, we obtain
\[
\begin{aligned}
F_{\rm ss}(x,y)
={}&(1+y)^4
+x\bigl(20+30y+18y^2+4y^3\bigr)\\
&+x^2\bigl(72+54y+12y^2\bigr)
+x^3(92+28y)
+39x^4.
\end{aligned}
\]
For the mixed orientation $Q_{\rm mix}$, we obtain
\[
\begin{aligned}
F_{\rm mix}(x,y)
={}&(1+y)^4
+x\bigl(20+28y+17y^2+4y^3\bigr)\\
&+x^2\bigl(76+50y+11y^2\bigr)
+x^3(101+26y)
+44x^4.
\end{aligned}
\]
These two $F$-triangles make the dependence on the orientation visible
at every face dimension and exhibit
more precisely how the two support $\tau$-tilting complexes differ.

\noindent{\bf Acknowledgements}. \quad  
This project is supported by  the National Natural Science Foundation of China (No. 12171207).

\bibliography{}

\begin{thebibliography}{9999}


\bibitem{A} T. Adachi, \emph{The classification of $\tau$-tilting modules over Nakayama algebras}, J. Algebra 452(2016), 227-262.

\bibitem{AA} T. Adachi and T. Aoki, \emph{The number of two-term tilting complexes over symmetric algebras with radical cube zero}, Ann. Comb. (1)27 (2023), 149-167.

\bibitem{AIR} T. Adachi, O. Iyama and I. Reiten, \emph{$\tau$-tilting theory}, Compos. Math., 150(3)(2014), 415-452.

\bibitem{ABS} I. Assem, T. Br{\"u}stle, R. Schiffler, \emph{Cluster tilted algebras as trivial extensions}, Bull. Lond. Math. Soc. 40 (1) (2008),  151-162


 \bibitem{AHR} I. Assem, D. Happel and O. Roldan, \emph{Representation-finite trivial extension algebras}, J. Pure Appl. Algebra 33 (1984), 235–242.

\bibitem{AMN} H. Asashiba, Y. Mizuno and K. Nakashima, \emph{Simplicial complexes and tilting theory for Brauer tree algebras}, J. Algebra 551 (2020), 119–153. 





\bibitem{ASS}	I. Assem, D. Simson and A. Skowro\'nski, \emph{Elements of the representation theory of associative algebras}. Vol. 1. Techniques of representation theory,	London Mathematical Society Student Texts 65, Cambridge Univ. Press Cambridge, 2006. 	
		
\bibitem{BR} E. Barnard and N. Reading, \emph{Coxeter-biCatalan combinatorics}, J. Algebraic Combin. 47 (2018), no. 2, 241–300.

\bibitem{BIRS} A. B. Buan, O. Iyama, I. Reiten, J. Scott, \emph{Cluster structures for 2-Calabi–Yau categories and unipotent
groups}, Compos. Math. 145(2009), 1035–1079.



\bibitem{Chapoton} F. Chapoton, \emph{Enumerative properties of generalized associahedra}, S\'em. Lothar. Combin. 51 (2004), Art. B51b.



\bibitem{DIJ} L. Demonet, O. Iyama and G. Jasso, \emph{$\tau$-tilting finite algebras, bricks and g-vectors}, Int. Math. Res. Not. 3(2019), 852-892. 



\bibitem{FGR} R. M. Fossum, P. A. Griffith and I. Reiten, \emph{Trivial extensions of abelian categories}, Lecture notes in mathematics 456, Springer-Verlag, 1975.


\bibitem{GH} H. Gao and Z. Huang, \emph{Silting modules over triangular matrix rings}, Taiwanese J. Math. 24(6)(2020), 1417-1437.




\bibitem{Happel} D. Happel, \emph{Triangulated Categories in the Representation of Finite Dimensional Algebras}, London Mathematical Society Lecture Note Series. Cambridge University Press, 1988.


 \bibitem{HW} D. Hughes and J. Waschbusch, \emph{Trivial extensions of tilted algebras}, Proc. London Math. Soc. 46 (1983), 347–364.

\bibitem{IZ} O. Iyama and X. Zhang, \emph{Classifying $\tau$-tilting modules over the Auslander algebra of $\Bbbk[x]/(x^n)$}, J. Math. Soc. Japan 72(3)(2020), 731-764.

\bibitem{LXZ} Z. Li, M. Xu and Z. Zhao, \emph{The $\tau$-rigid objects of a trivial extension of a hereditary abelian category}, Taiwanese J. Math. 29(5)(2025), 859-874.

\bibitem{LiZhang}  Z. Li and X. Zhang, \emph{$\tau$-tilting modules over trivial extensions}, Internat. J. Algebra Comput. 32(3)(2022),617-628.

\bibitem{Mizuno} Y. Mizuno, \emph{Classifying $\tau$-tilting modules over preprojective algebras of
Dynkin type}, Math. Z. 277(2014), no. 3-4, 665–690.



\bibitem{PMH} Y. Peng, X. Ma and Z. Huang, \emph{$\tau$-tilting modules over triangular matrix artin algebras}, Internat. J. Algebra Comput. 31(4)(2021), 639-661.

\bibitem{PPP} Y. Palu, V. Pilaud and P.-G. Plamondon, \emph{Non-kissing complexes and $\tau$-tilting for gentle algebras}, Mem. Am. Math. Soc. 274 (2021), no. 1343.





\bibitem{Tachikawa80} H. Tachikawa, \emph{Representations of trivial extensions of hereditary algebras}, in: Proc. ICRA II, in: Lecture
Notes in Math., vol. 832, Springer-Verlag, New York, 1980, 579–599.


\bibitem{WZ} Q. Wang and Y. Zhang, \emph{On $\tau$-tilting modules over trivial extensions of gentle tree algebras}, Kodai Math. J. 49 (1)(2026) 80 - 94.


\bibitem{XieZhang} Z. Xie and X. Zhang, \emph{A bijection theorem for Gorenstein projective $\tau$-tilting modules}, J. Algebra. Appl. 23(2024), no,10, article 2450168.



\bibitem{Zhang} Y. Zhang, \emph{Gluing support $\tau$-tilting modules via symmetric ladders of height 2}, Science China Math. 67(2024), 2217-2236.

\bibitem{Zito} S. Zito, \emph{$\tau$-tilting finite cluster-tilted algebras}, Proc. Edinb. Math. Soc. 63(4)(2020), 950-955.

\end{thebibliography}

\vskip 8pt

 {\footnotesize \noindent Rong Rong\\
 School of Mathematics and Statistics, Xuzhou University of Technology, Xuzhou, 221111, Jiangsu, PR China}
 
 { \footnotesize \noindent Zhi-Wei Li\\
School of Mathematics and Statistics, Jiangsu Normal University, Xuzhou 221116, Jiangsu, PR China\\

\end{document}